\documentclass[12 pt]{amsart}%
\usepackage{graphicx, color}
\usepackage{amscd}
\usepackage{amsmath}
\usepackage{amsfonts}
\usepackage{amssymb}%
\providecommand{\U}[1]{\protect\rule{.1in}{.1in}}
\providecommand{\U}[1]{\protect\rule{.1in}{.1in}}
\theoremstyle{plain}

\newtheorem{theorem}{Theorem}[section]
\newtheorem{proposition}[theorem]{Proposition}

\newtheorem{example}[theorem]{Example}
\newtheorem{remark}[theorem]{Remark}
\newtheorem{problem}[theorem]{Problem}

\newtheorem{lemma}[theorem]{Lemma}

\numberwithin{equation}{section}

\begin{document}
\title[Lineability of the set of bounded non-absolutely summing operators ]{Lineability of the set of bounded linear non-absolutely summing operators }
\author{Geraldo Botelho, Diogo Diniz and Daniel Pellegrino}
\thanks{The third named author is supported by CNPq Grant 308084/2006-3.}

\begin{abstract}
In this note we solve, except for extremely pathological cases, a
question posed by Puglisi and Seoane-Sep\'ulveda on the lineability
of the set of bounded linear non-absolutely summing operators. We
also show how the idea of the proof can be adapted to several
related situations.
\end{abstract}
\maketitle

\section{Introduction and notation}

Henceforth $E,$ $F$ and $G$ will stand for infinite-dimensional (real or complex) Banach spaces. The
topological dual of $F$ is represented by $F^{\ast}.$

According to \cite{Ar, En, gq} and others, a subset $A$ of an infinite-dimensional vector space $X$ is said to be {\it lineable}
if $A\cup\{0\}$ contains an infinite-dimensional subspace of $X$.

The space of absolutely $(r,s)$-summing linear operators from $E$ to $F$ will be
denoted by $\Pi_{r,s}(E;F)$ ($\Pi_{r}(E;F)$ if $r=s)$
and the space of bounded linear operators from $E$ to $F$ will be represented
by $\mathcal{L}(E;F).$ For details on the theory of absolutely summing
operators we refer to \cite{Diestel}.

Recently, D. Puglisi and J. Seoane-Sep\'{u}lveda \cite{seoane} proved, among other interesting results, that if
$E$ has the two series property and $G=F^{\ast}$ for some $F$, then the set%
\[
\mathcal{L}(E;G)\diagdown \Pi_{1}(E;G)
\]
is lineable. In the same paper the authors pose the following question:

\begin{problem}\label{problem}\rm
If $E$ is superreflexive and $p\geq1$, is it true that
\[
\mathcal{L}(E;F)\diagdown \Pi_{p}(E;F)
\]
is lineable for every Banach space $F$?
\end{problem}

M. A. Sofi, in a private communication to the authors, kindly
pointed out that the following situation should be settled first:
given operator ideals ${\mathcal I}_1$ and ${\mathcal I}_2$ and
Banach spaces $E$ and $F$, is it always true that
$\mathcal{I}_1(E;F)\diagdown {\mathcal I}_{2}(E;F)$ is either empty
of lineable? Quite surprisingly, we have:

\begin{example}\rm Let $SS$ denote the ideal of strictly singular linear operators and $E$ be an hereditarily indecomposable complex Banach space. Let us show that the set $\mathcal{L}(E;E)\diagdown SS(E;E)$, which is not empty because of the identity operator, does not contain a two dimensional subspace. Let $u_1, u_2$ be arbitrary linearly independent operators in  $\mathcal{L}(E;E)\diagdown SS(E;E)$. By \cite[Theorem 6]{maurey} there are scalars $\lambda_1, \lambda_2$ and strictly singular operators $v_1, v_2 \in SS(E;E)$ such that $u_1 = \lambda_1 id_E + v_1$ and $u_2 = \lambda_2 id_E + v_2$. It is clear that $\lambda_1 \neq 0\neq \lambda_2$ because $u_1$ and $u_2$ are not strictly singular. Letting $u = \lambda_2u_1 - \lambda_1u_2$ we have that $u \neq 0$ because $u_1$ and $u_2$ are linearly independent and from $u = \lambda_2v_1 - \lambda_1v_2$ we conclude that $u$ is strictly singular. Hence $u$ belongs to the subspace generated by $u_1$ and $u_2$ but $u \notin (\mathcal{L}(E;E)\diagdown SS(E;E)) \cup \{0\}$, proving that $(\mathcal{L}(E;E)\diagdown SS(E;E))\cup \{0\}$ does not contain a two dimensional subspace.
\end{example}

In the absence of a general result, particular situations must be investigated by {\it ad hoc} arguments.
The aim of this short note is to answer Problem \ref{problem} in the positive, except for very
particular quite pathological cases, and to extend the idea of the proof to related situations.

\bigskip

\section{Superreflexive spaces}

By $\mathcal K$ we denote de ideal of compact operators.

\begin{theorem}
\label{aaa}Let $p\geq1$ and $E$ be superreflexive. If either $E$ contains a
complemented infinite-dimensional subspace with unconditional basis or $F$ contains an infinite
unconditional basic sequence, then $\mathcal{K}%
(E;F)\diagdown \Pi_{p}(E;F)$ (hence $\mathcal{L}%
(E;F)\diagdown \Pi_{p}(E;F)$) is lineable.
\end{theorem}

\begin{proof} Assume that $E$ contains a complemented infinite-dimensional subspace $E_{0}$ with unconditional basis
$(e_{n})_{n=1}^{\infty}$. First consider%
\begin{equation}
\mathbb{N}=A_{1}\cup A_{2}\cup \cdots\label{bbbb}%
\end{equation}
a decomposition of $\mathbb{N}$ into infinitely many infinite
pairwise disjoint subsets $(A_j)_{j=1}^\infty$. Since
$\{e_{n}\,;\,n\in\mathbb{N}\}$ is an unconditional basis, it is well
known (e.g., combine \cite[Proposition 1.c.6]{Lin} and
\cite[Proposition 1.1.9]{Kalton}) that $\{e_{n}\,;\,n\in A_{j}\}$ is
an unconditional basic sequence for every $j \in \mathbb{N}$. Let us
denote by $E_{j}$ the closed span of $\{e_{n}\,;\,n\in A_{j}\}.$ As
a subspace of a superreflexive space, $E_j$ is superreflexive as
well, so from \cite[Theorem]{davis} it follows that for each
$j$ there is an operator%
\[
u_{j}\colon E_{j}\longrightarrow F
\]
belonging to ${\mathcal{K}}(E_{j};F)\diagdown \Pi_{p}(E_{j};F)$.

Denoting by $\varrho$ the unconditional basis constant of $(e_{n})_{n=1}^{\infty}$ we know that%
\[
\left\Vert
{\displaystyle\sum\limits_{j=1}^{\infty}}
\varepsilon_{j}a_{j}e_{j}\right\Vert \leq\varrho\left\Vert
{\displaystyle\sum\limits_{j=1}^{\infty}}
a_{j}e_{j}\right\Vert
\]
for every $\varepsilon_{j}=\pm 1$ and scalars $a_j$. For each $i$ we denote by $P_{i}\colon E_{0}\longrightarrow E_{i}$
the canonical projection onto $E_{i}$. For
\[
y=%
{\displaystyle\sum\limits_{j=1}^{\infty}}
a_{j}e_{j}\in E_{0}%
{\rm ~and~} x=P_{i}(y) \in E_i\] we have
\[
2x=%
{\displaystyle\sum\limits_{j\in A_{i}}}
2a_{j}e_{j}=%
{\displaystyle\sum\limits_{j=1}^{\infty}}
\varepsilon_{j}a_{j}e_{j}+%
{\displaystyle\sum\limits_{j=1}^{\infty}}
\varepsilon_{j}^{\prime}a_{j}e_{j}%
\]
for a convenient choice of signs $\varepsilon_{j}$ and $\varepsilon_{j}^{\prime}.$ Thus
\[
2\left\Vert P_{i}(y)\right\Vert =\left\Vert 2x\right\Vert \leq\left\Vert
{\displaystyle\sum\limits_{j=1}^{\infty}}
\varepsilon_{j}a_{j}e_{j}\right\Vert +\left\Vert
{\displaystyle\sum\limits_{j=1}^{\infty}}
\varepsilon_{j}^{\prime}a_{j}e_{j}\right\Vert \leq2\varrho\left\Vert
y\right\Vert .
\]
So each projection $P_{i}\colon E_{0}\longrightarrow E_{i}$ is continuous and has norm
$\leq\varrho$. This also implies that each $E_{i}$ is a complemented subspace
of $E_{0}$.

If $\pi_{0} \colon E\longrightarrow E_0$ denotes the projection onto $E_0$, for each $j \in \mathbb{N}$ we can define de operator%
\[
\widetilde{u_{j}} \colon E\longrightarrow F~,~\widetilde{u_{j}} := u_j \circ P_j \circ \pi_0.
\]
Since $(P_j \circ \pi_0)(x) = x$ for every $x \in E_j$, it is plain
that $\widetilde{u_{j}}$ belongs to $\mathcal{K}(E;F)\diagdown
\Pi_{p}(E;F)$. Given $n \in \mathbb{N}$ and scalars $a_1, \ldots,
a_n$, with at least one $a_k \neq 0$, $1 \leq k \leq n$, since
$\widetilde{u_{k}}$ fails to be absolutely $p$-summing, there is a
weakly $p$-summable sequence $(x_j)$ in $E_k$ such that $\sum_j
\|{u_{k}}(x_j)\|^p = + \infty$. It is clear that $(x_j)$ is weakly
$p$-summable in $E$ and $\widetilde{u_{k}}(x_j) = u_k(x_j)$ for
every $j$. But $A_k \cap A_i = \emptyset$ for
 $i = 1, \ldots, n$, $i \neq k$, so it follows that $\widetilde{u_{i}}(x_j) = 0$ for every $i = 1, \ldots, n$, $i \neq k$ and $j \in \mathbb{N}$.
 So,
 $$\sum_j \|a_1 \widetilde{u_{1}}(x_j) + \cdots + a_n \widetilde{u_{n}}(x_j)\|^p =  \sum_j \|a_k{u_{k}}(x_j)\|^p = + \infty,$$
 proving that $a_1 \widetilde{u_{1}} + \cdots + a_n \widetilde{u_{n}}$ is not absolutely $p$-summing. This proves that
 the span of $\{\widetilde{u_{j}};j\in\mathbb{N}\}$ is contained in $\mathcal{K}(E;F)\diagdown
\Pi_{p}(E;F)$.

Let us see now that the set $\{\widetilde{u_{j}};j\in\mathbb{N}\}$ is linearly independent. Let $n \in \mathbb{N}$ and $a_1, \ldots, a_n$ be scalars such that
$$a_1\widetilde{u_{1}} + \cdots + a_n \widetilde{u_{n}} = 0. $$
For every $k \in \{1, \ldots, n\}$ we can choose $x_k \in E_k$ such that $\widetilde{u_{k}}(x_k) \neq 0$ because $\widetilde{u_{k}} \neq 0$. But $(P_j \circ \pi_0) (x_k) =  P_j(x_k) = 0$ for every $j = 1, \ldots, n$, $j \neq k$. So,
$$a_k\widetilde{u_{k}}(x_k) = 0 + \cdots + 0 + a_k\widetilde{u_{k}}(x_k) + 0 + \cdots 0 = a_1\widetilde{u_{1}}(x_k) + \cdots + a_n \widetilde{u_{n}}(x_k) = 0. $$
It follows that $a_k = 0$. Hence the span of
$\{\widetilde{u_{j}};j\in\mathbb{N}\}$ is an infinite-dimensional
subspace contained in $\mathcal{K}(E;F)\diagdown \Pi_{p}(E;F)$.

Now, suppose that $F$ contains a subspace $G$ with unconditional basis
$\{e_{n};n\in\mathbb{N}\}$ with unconditional basis constant $\varrho$. Still considering
the subsets $(A_{n})$ of $\mathbb{N}$ as above, define $F_{j}$ as the closed span of $\{e_{n};n\in
A_{j}\}$ and let $P_j \colon G \longrightarrow F_j$ be the corresponding projections. Proceeding as above we conclude that
$\|P_j\| \leq \varrho$. From \cite[Theorem]{davis} we know that for each $j$ there is an operator%
\[
u_{j} \colon E\longrightarrow F_{j}%
\]
belonging to $\mathcal{K}(E;F_{j})\diagdown \Pi_{p}(E;F_{j}).$

Recall that $F_{i}\cap F_{j}=\{0\}$ if $i\neq j.$ So, if $y_{i}\in F_{i}$ and
$y_{j}\in F_{j}$ (with $i\neq j$), we have%
\begin{equation}
\left\Vert y_{i}\right\Vert =\left\Vert P_{i}(y_{i}+y_{j})\right\Vert
\leq\varrho\left\Vert y_{i}+y_{j}\right\Vert . \label{mmm}%
\end{equation}
Now by $\widetilde{u_{j}}$ we mean the composition of $u_{j}$ with
the inclusion from $F_{j}$ to $F$. It is clear that
$\widetilde{u_{j}}$ is compact and fails to be absolutely
$p$-summing. From
(\ref{mmm}) it follows that%
\[
\left\Vert \widetilde{u_{i}}(x)+\widetilde{u_{j}}(x)\right\Vert \geq
\varrho^{-1}\left\Vert \widetilde{u_{i}}(x)\right\Vert
\]
for every $x\in E.$ Hence%
\[
\widetilde{u_{i}}+\widetilde{u_{j}}\in \mathcal{K}(E;F)\diagdown
\Pi_{p}(E;F){\rm ~for~all~} i, j,
\]
and so we can easily deduce that the span of
$\{\widetilde{u_{j}};j\in \mathbb{N}\}$ is contained in
$(\mathcal{K}(E;F)\diagdown$\\$ \Pi_{p}(E;F))\cup\{0\}$. A reasoning
similar to the first case shows that the vectors
$\widetilde{u_{j}}$,
 $j \in \mathbb{N}$, are linearly independent, therefore $\mathcal{K}(E;F)\diagdown \Pi_{p}(E;F)$ is lineable.
\end{proof}

\begin{remark}\rm
Note that Theorem \ref{aaa} solves the problem posed by Puglisi and
Seoane-Sep\'{u}lveda except when $E$ is a superreflexive Banach
space not containing an infinite-dimensional complemented subspace
with unconditional basis (such a space was constructed by V.
Ferenczi \cite{Ferenczi, F2}) {\it and} $F$ does not contain an
infinite-dimensional subspace with unconditional basis (for example,
hereditarily indecomposable spaces). It is in this sense we claim
that Theorem \ref{aaa} solves the problem modulo extremely
pathological cases.
\end{remark}

\medskip

\section{Non necessarily superreflexive spaces}

Examining the proof of Theorem \ref{aaa} it becomes clear that the
result holds if: (i) $E$ contains a sequence
$(E_{n})_{n=1}^{\infty}$ of complemented infinite-dimensional
subspaces such that $E_{n}\cap E_{m}=\{0\}$ if $m\neq n$; (ii)
$\mathcal{L}(E_n;F)\diagdown \Pi_{p}(E_n;F)\neq \emptyset$ for every
$n \in \mathbb{N}$. Having this in mind, the argument of the proof
can be adapted to many other circumstances, even for spaces of
operators on non-superreflexive spaces.

We start by adapting the proof of Theorem \ref{aaa} to spaces of
operators on spaces containing complemented copies of $\ell_{1}$ or
$c_{0}$ (observe that in these cases the domain spaces are not even
reflexive):

\medskip

\begin{proposition}~\\
{\rm (a)} If $E$ contains a complemented copy of $\ell_1$ and $F$ is
not isomorphic to a Hilbert space, then $\mathcal{L}(E;F)\diagdown
\Pi_{1}(E;F)$ is lineable.\\
{\rm (b)} If $E$ contains a complemented copy of $c_0$ and $1\leq
p<2$, then $\mathcal{L}(E;F)\diagdown \Pi_{p}(E;F)$ is lineable for
every Banach space $F$.
\end{proposition}

\begin{proof} Up to the composition with the corresponding projections, it suffices to work with $E = \ell_1$ in (a) and
$E= c_0$ in (b).\\
(a) Decomposing $\mathbb{N}$ as in (\ref{bbbb}) we have that the
closed span
of each%
\[
\{e_{n};n\in A_{j}\},
\]
denoted by $E_j$, is a complemented copy of $\ell_{1}$ which is isometrically isomorphic to $\ell_{1}%
$. From \cite{LP} we know that
\[
\mathcal{L}(\ell_{1};F)\diagdown \Pi_{1}(\ell_{1};F)\neq \emptyset,
\]
so
\[
\mathcal{L}(E_n;F)\diagdown \Pi_{1}(E_n;F)\neq \emptyset {\rm ~for~every~~} n.
\]
Now proceed as in the proof of Theorem \ref{aaa} to complete the proof.

\medskip

\noindent (b) Using that $c_0$ enjoys the same property of $\ell_1$ we used above and that
\[
\mathcal{L}(c_{0};F)\diagdown \Pi_{p}(c_{0};F)\neq \emptyset {\rm ~for~every~} F,
\]
see \cite{Bot-Pell, st}, the proof of (a) can be repeated line by line.
\end{proof}

The proof of Theorem \ref{aaa} also makes clear that the result
holds if: (i) $F$ contains a sequence $(F_{n})_{n=1}^{\infty}$ of
infinite-dimensional subspaces such that $F_{n}\cap F_{m}=\{0\}$ if
$m\neq n$; (ii) $\mathcal{L}(E;F_n)\diagdown \Pi_{p}(E;F_n)\neq
\emptyset$ for every $n \in \mathbb{N}$. An adaptation of this case
yields, for example:

\begin{proposition} If $p\geq1$, then $\mathcal{L}(E;F)\diagdown \Pi_{p}(E;F)$ is lineable for every Banach space $E$ and every Banach space
$F$ containing a copy of $c_0$.
\end{proposition}

\begin{proof} From \cite{belg} we know that
\[
\mathcal{L}(E;c_{0})\diagdown \Pi_{p}(E;c_{0})\neq \emptyset {\rm ~for~every~}E.
\]
Using again that $c_{0}$ has infinitely many \textquotedblleft
independent\textquotedblright\ copies of itself, the idea of the proof of
Theorem \ref{aaa} provides the result.
\end{proof}

\medskip

\section{Non-absolutely $(q,1)$-summing linear operators}

In this section we turn our attention to the lineability of the set
of non-absolutely $(q,1)$-summing operators, which is, {\it a
priori}, a more delicate matter. Absolutely $(q,1)$-summing
operators are closely connected to the cotypes of the underlying
spaces; for this reason, given a Banach space $F$, we define $\cot F
= \inf \{q \geq 2 : F {\rm ~has~cotype~}q\}$.\\
\indent As before, to prove the lineability of the set
$\mathcal{L}(E;F)\diagdown \Pi_{q,1}(E;F)$, a non-coincidence result
is needed to start the
process. A result from \cite{Bot-Pell} will serve this purpose. If $E$ has unconditional basis $(x_{n})_{n=1}^{\infty}$, define%
\[
\mu_{E,(x_{n})}=\inf \mbox{\Large\{}\,t:(a_{j})_{j=1}^{\infty}\in \ell_{t}\text{ whenever }x=%
{\displaystyle\sum\limits_{j=1}^{\infty}}
a_{j}x_{j}\in E \mbox{\Large\}}.
\]

\begin{proposition}\label{prop}{\rm (\cite[Corollary 2.1]{Bot-Pell})}
If $q<\cot F,$ $E$ has an unconditional normalized basis
$(x_{n})_{n=1}^{\infty}$ and $\mu_{E,(x_{n})}>q$, then%
\[
\mathcal{L}(E;F)\diagdown \Pi_{q,1}(E;F)\neq \emptyset.
\]
\end{proposition}

By adapting the arguments we used so far with Proposition \ref{prop}
as starting point, it is not difficult to prove that:

\begin{proposition}\label{pro} If $1 \leq q < {\rm cot} F$ and $p > q$, then $\mathcal{L}(\ell_{p};F)\diagdown \Pi_{q,1}(\ell_{p};F)$ is lineable.
\end{proposition}

We shall improve substantially both Proposition \ref{pro} (in the
sense that $\ell_p$ can be replaced by spaces $E$ with unconditional
basis $(x_n)$ such that $\mu_{E,(x_{n})}>q$) and Proposition
\ref{prop} (in the sense that $\mathcal{L}(E;F)\diagdown\Pi%
_{q,1}(E;F)$ is actually lineable). We will need the following
result:

\begin{lemma}{\rm (\cite[Lemma 1.1]{seoane})}
\label{ss}Let $(a_{n})_{n=1}^{\infty}$ be a sequence of positive real numbers.
If $%
{\displaystyle\sum\limits_{j=1}^{\infty}}
a_{n}=\infty$, then there is a sequence of sets of positive integers
$(A_{j})_{j=1}^{\infty}$ so that:

\noindent {\rm (i)} $\mathbb{N}=A_{1}\cup A_{2}\cup \cdots$.

\noindent {\rm (ii)} Each $A_{j}$ has the same cardinality of $\mathbb{N}$.

\noindent {\rm (iii)} The sets $A_{j}$ are pairwise disjoint.

\noindent {\rm (iv)} $%
{\displaystyle\sum\limits_{j\in A_{k}}}
a_{j}=\infty$ for each $k$.
\end{lemma}

\begin{theorem}\label{prop2}
If $1\leq q<\cot F,$ $E$ has an unconditional normalized basis
$(x_{n})_{n=1}^{\infty}$ and $\mu_{E,(x_{n})}>q$,
then%
\[
\mathcal{L}(E;F)\diagdown \Pi_{q,1}(E;F)
\]
is lineable.
\end{theorem}

\begin{proof} Since $\mu_{E,(x_{n})}>q$, we can find $(a_{j})_{j=1}^{\infty}$ and $\varepsilon >0$ so that%
\begin{equation}%
x = {\displaystyle\sum\limits_{j=1}^{\infty}}
a_{j}x_{j}\in E {\rm ~~~~ and ~~~~ } {\displaystyle\sum\limits_{j=1}^{\infty}}
\left\vert a_{j}\right\vert ^{q+ \varepsilon}=\infty.\label{x}%
\end{equation}
Let $(A_{j})_{j=1}^{\infty}$ be the sets of Lemma \ref{ss}
associated to the divergent series $
{\displaystyle\sum\limits_{j=1}^{\infty}} \left| a_{j}\right| ^{q+
\varepsilon}$. For each positive integer $k,$ define
\[
E_{k}=\overline{span\{x_{j};j\in A_{k}\}}.
\]
From the proof of Theorem \ref{aaa} we know that each $\{x_{n};n\in A_{k}\}$ is an unconditional basic
sequence and $E_k$ is a complemented
subspace of $E$. From
the choice of $A_{k}$ we have that $\mu_{E_{k},(x_{n})}>q$, so Proposition \ref{prop} gives that
\[
\mathcal{L}(E_{k};F)\diagdown \Pi_{q,1}(E_{k};F)\neq\emptyset
\]
for every $k$. Since each $E_{k}$ is a complemented subspace of $E$ the result
follows by repeating once more the procedure of the proof of Theorem \ref{aaa}.
\end{proof}

\bigskip

\noindent {\bf Acknowledgements.} The authors thank V. Ferenczi and M. A. Sofi for helpful conversations on the subject of this paper.

\bigskip\noindent[Geraldo Botelho] Faculdade de Matem\'atica, Universidade
Federal de Uberl\^andia, 38.400-902 - Uberl\^andia, Brazil, e-mail: botelho@ufu.br.

\medskip

\noindent[Diogo Diniz] IMECC- UNICAMP, Caixa Postal 6065, Campinas, SP, Brazil, e-mail:
diogodpss@gmail.com.

\medskip

\noindent[Daniel Pellegrino] Departamento de Matem\'atica, Universidade
Federal da Para\'iba, 58.051-900 - Jo\~ao Pessoa, Brazil, e-mail: dmpellegrino@gmail.com.

\end{document}